\documentclass[12pt]{article}

\usepackage[demo]{graphicx} 
\usepackage{hyperref}
\usepackage{float}
\usepackage{indentfirst}
\usepackage{tikz}
\usepackage{amssymb}
\usepackage{amsmath,mathrsfs,amsfonts}
\usepackage{graphics}
\usepackage{amsthm}
\usepackage{setspace}
\usepackage{cases}
\usepackage{setspace}
\usepackage[mathscr]{eucal}
\usepackage{caption}
\usepackage{subcaption}
\usepackage{rotating}
\usetikzlibrary{decorations.pathreplacing}
\newtheorem{theo}{Theorem}[section]
\newtheorem{theorem}[theo]{Theorem}
\newtheorem{lem}[theo]{Lemma}

\newtheorem{corollary}[theo]{Corollary}

\newtheorem{conj}[theo]{Conjecture}

 \theoremstyle{definition}
\newtheorem{defi}[theo]{Definition}

\newtheorem{fact}[theo]{Fact}
\newtheorem{ques}[theo]{Question}

\theoremstyle{remark}

\newcounter{casenum}[theo]

\newcounter{subcasenum}[theo]

\newcounter{claimnum}[theo]

\allowdisplaybreaks
\begin{document}
\thispagestyle{plain}

\begin{center} {\Large Three Results on   Generalized Quasikernels in Digraphs
}
\end{center}
\pagestyle{plain}
\begin{center}
	{
		{\small  Zejun Huang\footnote{  Corresponding author. \\ Email: zejunhuang@szu.edu.cn (Huang), yangchenxi2022@email.szu.edu.cn (Yang)}, Chenxi Yang}\\[3mm]
		{\small   School of Mathematical Sciences, Shenzhen University, Shenzhen 518060, China }\\
		
	}
\end{center}
\begin{center}
\begin{minipage}{140mm}
\begin{center}
{\bf Abstract}
\end{center}
{\small
A $q$-kernel of a digraph $D$ is an independent set $Q\subseteq D$ such that every vertex of $D$ is reachable from $Q$ by a directed path of length at most $q$, which is a natural generalization of kernels and quasikernels. In this paper, we establish three results on generalized quasikernels. Firstly, we prove that any $n$-vertex source-free bipartite  oriented graph with no directed 4-cycles has a quasikernel of size at most $17n/35$. Secondly, we show that every digraph with no $(r-1)$-source set  contains $r$ pairwise disjoint $(3r-2)$-kernels, where $r\ge 2$. At last, we consider unicyclic digraph  with a directed cycle of length $2\ell$ and bipartition $U\cup V$, and we prove that for every odd integer $q\ge 3$, there exist two $q$-kernels $Q_U\subseteq U$ and $Q_V\subseteq V$
 such that \[
|Q_U|+|Q_V| \le 2\cdot \frac{\lceil \ell/(q+1)\rceil}{\ell} |V(D)|.
\]  These results confirm two conjectures and give an affirmative answer to a question posed by Spiro in European Journal of Combinatorics 133 (2026), 104307.

}

\end{minipage}
\end{center}

\section{Introduction and main results}
We use standard digraph terminology. Throughout this paper, all digraphs are finite and
have no loops or multiple arcs. For vertices $u,v$ in a digraph $D$, we denote by
$\operatorname{dist}_D(u,v)$ the length of a shortest directed path from $u$ to $v$; if no
such path exists, we set $\operatorname{dist}_D(u,v)=\infty$. For a digraph $D$ and a set
$S\subseteq V(D)$, we denote by $D[S]$ the subgraph of $D$ induced by $S$. A set
$S\subseteq V(D)$ is called \emph{independent} if there is no arc between any two vertices
of $S$. For $S\subseteq V(D)$, we define
$\operatorname{dist}_D(S,v):=\min_{s\in S}\operatorname{dist}_D(s,v).$

For $S\subseteq V(D)$, we also write
$N_D^-(S)=\{x\in V(D)\setminus S:\ \text{there exists } y\in S \text{ such that } xy\in A(D)\}$, i.e., $N_D^-(S)$ is the set of external in-neighbors of $S$. For a vertex $v\in V(D)$,
we write $N_D^-(v):=N_D^-(\{v\})$. A digraph is called \emph{source-free} if
$N_D^-(v)\ne\emptyset$ for every $v\in V(D)$. When $D$ is clear from the context, we
simply write $N^-(S)$ and $N^-(v)$. For a digraph $D$, $v\in V(D)$, and an integer
$t\ge 1$, the closed backward $t$-ball centered at $v$ is defined as
\[
B_D^-(v,t):=\{u\in V(D)\mid \operatorname{dist}_D(u,v)\le t\}.
\]
When $D$ is clear from the context, we simply write $B^-(v,t)$.

Let $\mathcal{C}(D)$ be the set of strongly connected components of a digraph $D$. The
condensation digraph of $D$, denoted by $\operatorname{con}(D)$, is the digraph with
$V(\operatorname{con}(D))=\mathcal{C}(D)$
and
\[
A(\operatorname{con}(D))=\{\, (C,C')\in \mathcal{C}(D)\times \mathcal{C}(D): C\ne C'
\text{ and there exist} \, u\in C, v\in C' \text{ such that } uv\in A(D)\,\}.
\]
A strongly connected component $C$ of $D$ is a \emph{source component} if the
corresponding vertex in $\operatorname{con}(D)$ has in-degree $0$. A topological ordering
of a finite digraph $D$ is an ordering of its vertices $v_1,\dots,v_n$ such that for every
arc $v_i v_j\in A(D)$, we have $i<j$.

Let $K\subseteq V(D)$ be an independent set of $D$. We call $K$   a \emph{kernel}, a \emph{quasikernel}, and a \emph{$q$-kernel} of $D$ if
$\operatorname{dist}_D(K,v)\le 1$, $\operatorname{dist}_D(K,v)\le 2$, and $\operatorname{dist}_D(K,v)\le q$ for every $v\in V(D)$, respectively.
Thus a kernel is a $1$-kernel, and a quasikernel is a $2$-kernel. Kernels and quasikernels are classical absorbing structures in digraphs. While kernels do
not exist in all digraphs, the following result shows that quasikernels always exist.

\begin{theorem}\label{qk}\cite{C}
Every digraph contains a quasikernel.
\end{theorem}

 This
qualitative existence theorem naturally leads to quantitative questions, the most prominent
of which asks how small a quasikernel can be. In 1976, Erd\H{o}s and
Sz\'ekely proposed the following conjecture.

\begin{conj}\label{conj:small-qk}\cite{E}
Every $n$-vertex source-free digraph has a quasikernel of size at most $n/2$.
\end{conj}

 Gutin et al. \cite{GKTY} studied a natural strengthening of
Conjecture~\ref{conj:small-qk}: if two disjoint quasikernels exist, then one of them has
size at most half of the vertex set. In the terminology of this paper, they showed that a
source-free digraph with a quasikernel of size at most two has two disjoint quasikernels,
and they also constructed source-free digraphs with no pair of disjoint quasikernels.
Heard and Huang \cite{HH} showed that every source-free digraph $D$ has two
disjoint quasikernels when $D$ is semicomplete multipartite, quasi-transitive, or locally
semicomplete. Consequently, the small quasikernel conjecture holds for these three classes
of digraphs. Further positive results were obtained by Kostochka, Luo and Shan \cite{KLS},
who proved the conjecture for a broad class containing all orientations of $4$-colorable
graphs. Ai et al. \cite{AG} proved that the conjecture also holds for source-free
claw-free digraphs, and for source-free digraphs with no induced copy of
$\overrightarrow{K}_{4,1}$ or $\overrightarrow{K}^{+}_{4,1}$.

Recently, Spiro \cite{S} initiated the study of extremal problems for generalized
quasikernels. He proved that if $D$ is a source-free digraph and
$q\ge 3$, then $D$ contains two disjoint $q$-kernels; if a digraph has no $(r-1)$-source sets, then it contains $r$ disjoint $2^{r+1}$-kernels.
He also proved that if $D$ is a
source-free bipartite digraph which is not a disjoint union of directed $2$-cycles and
directed $4$-cycles, then $D$ contains a quasikernel $Q$ with
$|Q|<|V(D)|/2$. Furthermore, he proposed the following conjecture.
\begin{conj}\cite{S}
There exists $\varepsilon>0$ such that if $D$ is a source-free bipartite
digraph with no directed $2$-cycle and no directed $4$-cycle, then $D$
has a quasikernel $Q$ with
\[
|Q|\le \left(\frac12-\varepsilon\right)|V(D)|.
\]
\end{conj}

A non-empty set $S\subseteq V(D)$ is called an \emph{$s$-source set} if
\[
N^-(S)=\emptyset \quad\text{and}\quad |S|\le s.
\] A digraph is unicyclic if every vertex has in-degree 1 and   its
underlying graph is connected.
In the same paper of  Spiro \cite{S}, he proposed the following conjecture and question.

\begin{conj}\cite{S}
There exist $\varepsilon,r_0>0$ such that if $r\ge r_0$ and $D$ has no $(r-1)$-source set, then $D$ contains $r$ pairwise disjoint $q$-kernels
with
\[
q\le (2-\varepsilon)^r.
\]
\end{conj}

\begin{ques}\cite{S}\label{question}
Let $D$ be a  unicyclic digraph  with a directed cycle of length $2\ell$ and bipartition $U\cup V$.
When $q\ge 3$ is odd, are there  $q$-kernels $Q_U\subseteq U$ and $Q_V\subseteq V$ such that
\[
|Q_U|+|Q_V| \le 2\cdot \frac{\lceil \ell/(q+1)\rceil}{\ell} |V(D)|?
\]
\end{ques}

In this paper, we prove the above conjectures and answer Question \ref{question} affirmatively. Our main results are as follows.

\begin{theorem}\label{th3}
Let $D$ be an $n$-vertex source-free bipartite oriented graph without directed
$4$-cycles. Then $D$ has a quasikernel $Q$ such that
\[
|Q|\le \frac{17}{35}n.
\]
\end{theorem}
\begin{theorem}\label{th1}
Let $r\ge 2$.  If a digraph $D$ has no $(r-1)$-source set, then $D$ contains $r$ pairwise disjoint $(3r-2)$-kernels.
\end{theorem}

\begin{theorem}\label{th2}
Let $D$ be a  unicyclic digraph  with a directed cycle of length $2\ell$ and bipartition $U\cup V$.
If $q\ge 3$ is odd, then there exist $q$-kernels $Q_U\subseteq U$ and
$Q_V\subseteq V$ such that
\[
        |Q_U|+|Q_V|
        \le
        2\cdot
        \frac{\lceil \ell/(q+1)\rceil}{\ell}|V(D)|.
\]
\end{theorem}

Theorem \ref{th3} confirms Conjecture 1.3 with $\varepsilon=1/70$. Theorem \ref{th1} confirms Conjecture 1.4 with $\varepsilon=1/2$ and $r_0=8$. We will present the proofs of Theorem \ref{th3}, Theorem \ref{th1} and Theorem \ref{th2} in Section 2, Section 3 and Section 4, respectively.

\section{Proof of Theorem \ref{th3}}

In this section, we always assume that  $D$ is a source-free bipartite digraph of order $n$ with bipartition $V(D)=A\cup B$. We need the following preliminary lemmas.
\begin{lem}\label{bck}\cite{RM}
Every finite bipartite digraph has a kernel.
\end{lem}
\begin{lem}\label{fact3}
Let $C$ be a directed cycle of even length $k \ge 6$. Then $C$ has a quasikernel $Q_C$
such that $|Q_C| \le \lceil k/3 \rceil\le 2k/5$, and no two consecutive vertices of $C$ lie in $Q_C$.
\end{lem}
\begin{proof}
Label the vertices of $C$ by $[k]$ in cyclic order. Denote by $t=\lceil k/3\rceil-1$.
Let $Q_C=\{1,4,7,\ldots,3t+1\}$ if $k\ne 3t+1$ and $Q_C=\{1,4,7,\ldots,3t-2,3t\}$ if $k=3t+1$.
Then $Q_C$ is a quasikernel with size $|Q_C| = \lceil k/3\rceil$.
For $t\ge 6$ is even, $\lceil t/3 \rceil\le 2t/5$ is always true.
\end{proof}

\begin{lem}\label{lemma7}
Let $X\subseteq A$.  If
\[
X\cap N_D^-(b)\ne \emptyset \quad \text{for all}\quad b\in B,
\]
then $X$ is a quasikernel of $D$.
\end{lem}
\begin{proof}
The set $X$ is independent because $X\subseteq A$.  Every $b\in B$ has
an in-neighbor in $X$, so $\operatorname{dist}_D(X,b)\le 1$.
Let $a\in A\setminus X$.  Since $D$ is source-free, $a$ has an
in-neighbor $b\in B$. By the assumption, there exists $x\in X\cap N_D^-(b)$.
Therefore,  we have $\operatorname{dist}_D(X,a)\le 2$, and $X$ is a quasikernel of $D$.
\end{proof}

\begin{lem}\label{lemma8}
Let $H$ be a simple graph with $m$ vertices and $e$ edges. Then $H$ has a vertex cover of size at most $(m+e)/3$.
\end{lem}

\begin{proof}
We use induction on $m$. For the base case $m=0$, since $H$ is an empty graph, the statement is trivial.
Now suppose $H$ is a simple graph with $m\ge 1$ edges and assume the statement holds for graphs with less than $m$ edges.

If $\Delta(H)\le 1$, then the edges of $H$ form a matching.
By selecting exactly one endpoint from each edge, we obtain a vertex cover whose size equals the number of edges
 $e\le m/2\le (m+e)/3$.

For the case $\Delta(H)\ge 2$, choose a vertex $v$ with degree $d\ge 2$.  By induction, $H-v$
has a vertex cover $U$ of size at most $\big((m-1)+(e-d)\big)/3$. It follows that $U\cup\{v\}$ is
  a vertex cover of $H$ of size at most $(m+e)/3$.
\end{proof}

\begin{defi}
A vertex $x\in V(D)$ is called \emph{good} if it has an out-neighbor
$y$ such that $N_D^-(y)=\{x\}$.
Such a vertex $y$ is called a \emph{private out-neighbor} of $x$. Let
\[ P_A=\{a\in A:a\text{ is good}\},\quad P_B=\{b\in B:b\text{ is good}\},\] and denote
$p_A=|P_A|,p_B=|P_B|.$

\end{defi}

\begin{lem}\label{lemma9}
$D$ has a quasikernel $Q_A\subseteq A$ satisfying
$|Q_A|\le (n+p_A)/3$.
Symmetrically, \(D\) has a quasikernel $Q_B\subseteq B$ satisfying
$|Q_B|\le (n+p_B)/3$.
\end{lem}

\begin{proof}
We only   prove the first part, since a symmetric argument works for the second part.
For each $a\in P_A$, choose one private out-neighbor $\phi(a)\in B$. By definition,   $\phi(a)\ne \phi(b)$ for distinct $a,b\in P_A$.

Let
\[
B_0=\{b\in B: N_D^-(b)\cap P_A=\emptyset\}.
\]
Then $|B_0|\le |B|-p_A$, as  $\{\phi(a): a\in P_A\}\subseteq B$ is disjoint with $B_0$.

For each $b \in B_0$, we have $|N_D^-(b)|\ge 2$; otherwise $N^{-}_D(b)=\{a\}$ leads to $a\in P_A$, which contradicts $b\in B_0$.
We construct a simple undirected graph $H$ with vertex set $V(H) = A \setminus P_A$.
For each $b \in B_0$, we select two distinct vertices from $N_D^-(b)$ and add an edge joining them in $H$, while repeated edges are ignored if they occur.
Obviously, we have $|V(H)|=|A|-p_A$ and $|E(H)|\le |B_0|\le |B|-p_A$.
By Lemma \ref{lemma8}, $H$ has a vertex cover $C$ with
\[
|C|\le \frac{(|A|-p_A)+(|B|-p_A)}{3}=\frac{n-2p_A}{3}.
\]

Let $Q_A=P_A\cup C$. If $b\notin B_0$, then $N_D^-(b)\cap P_A\ne \emptyset$.
If $b\in B_0$, then the edge chosen inside $N_D^-(b)$ is covered by $C$, so $C\cap N_D^-(b)\ne \emptyset$.
By Lemma \ref{lemma7}, $Q_A$ is a quasikernel.  Moreover,
$|Q_A|\le p_A+(n-2p_A)/3=(n+p_A)/3$.
\end{proof}

\begin{lem}\label{claim3}
$D$ has a quasikernel $Q$ with
\[
|Q|\le\frac{7n}{5}-p_A-p_B.
\]
\end{lem}

\begin{proof}
For every good vertex $x \in P_A\cup P_B$, choose one private out-neighbor and denote it by $\phi(x)$.
The map $\phi: P_A\cup P_B \to V(D)$ is injective.
Let $D'$ be the spanning subgraph of $D$ with arc set
\[
A(D')=\{x\phi(x)|x\in P_A\cup P_B \}.
\]
In $D'$, every vertex has out-degree at most $1$ and in-degree at most $1$.
Therefore, every weak connected component of $D'$ is a directed path or a directed cycle.
Since $|V(D')|-|A(D')|=n-p_A-p_B$,
there are exactly $n-p_A-p_B$ directed path components,
where an isolated vertex is regarded as a path of length zero.

Recall the assumption that $D$ is bipartite and contains no directed $2$-cycle or $4$-cycle.
For every directed cycle component $C$ of $D'$, by  Lemma \ref{fact3} we can choose a  quasikernel $Q_C$
such that $|Q_C| \le 2|V(C)|/5$, and no two consecutive vertices of $C$ lie in $Q_C$.
For every directed path component $P: x_1\to x_2\to\cdots\to x_\ell$
of $D'$, define
\[
Q_P=\{x_3,x_6,x_9,\dots,x_{3\lfloor \ell/3\rfloor}\}.
\]
Let
\[
Q_0 = \left( \bigcup_{C \in \mathcal{C}} Q_C \right) \cup \left( \bigcup_{P \in \mathcal{P}} Q_P \right),
\]
where $\mathcal{C}$ and $\mathcal{P}$ denote the families of directed cycle and directed path components of $D'$.
According to the choice of $Q_C$ and $Q_P$, we have
\[
|Q_0|\le \frac{2n}{5}.
\]

Let $S$ be the set of starting vertices $s$ of directed path components of $D'$ such that
\[
N_D^-(s)\cap Q_0=\emptyset.
\]
Since the number of directed path components is $n-p_A-p_B$, we have
$|S|\le n-p_A-p_B$. The induced digraph $D[S]$ is bipartite. By Lemma \ref{bck}, there exist a kernel $K$ of $D[S]$.
Define $Q=Q_0\cup K$. Then
\[
|Q|\le |Q_0|+|K|\le \frac{7n}{5}-p_A-p_B.
\]

Now we prove that $Q$ is independent. For every $q\in Q_0$, there exists a unique
$p\in P_A\cup P_B$ such that $pq\in A(D')$. Then we have $q=\phi(p)$ and
$N_D^-(q)=\{p\}$; we denote $\phi^{-1}(q):=p$. By the construction of $Q_P$
and $Q_C$, we have $\phi^{-1}(q)\notin Q_0\cup S$ for all $q\in Q_0$. Since
$K\subseteq S$, it follows that $\phi^{-1}(q)\notin Q$ for all $q\in Q_0$.
Hence there is no arc from $Q$ to $Q_0$.

Recall  that $N_D^-(s)\cap Q_0=\emptyset$ for all $s\in S$. Since
$K\subseteq S$, there is no arc from $Q_0$ to $K$. Together with the fact that
$K$ is a kernel of $D[S]$, and hence independent, we conclude that $Q$ is
independent.

It remains to show that $\operatorname{dist}_D(Q,x)\le 2$ for all $x\in V(D)$.
All vertices in the directed cycle components of $D'$ can be reached within two steps from the corresponding sets $Q_C$.
In a directed path component
$
P:x_1\to x_2\to\cdots\to x_\ell,
$
all vertices $x_i$ with $i\ge 3$ are reached within two steps from $Q_P$.
If $x_1\notin S$, by the definition of $S$,
there is some $q\in Q_0$ satisfies $q\to x_1$, so $q$ reaches $x_1$ in one step and $x_2$ in two steps if $x_2$ exists.
If $x_1\in K$, then $x_1$ is reached in zero steps and $x_2$ in one step if it exists.
If $x_1\in S\setminus K$, then because $K$ is a kernel of $D[S]$, there is
some $k\in K$ such that $k\to x_1$.
Hence $k$ reaches $x_1$ in one step and $x_2$ in two
steps if it exists.  Thus $Q$ is a quasikernel of $D$.
\end{proof}

Now we are ready to present the proof of   Theorem \ref{th3}.
\begin{proof}[Proof of Theorem \ref{th3}]

If $p_A\le \alpha n$,
applying Lemma \ref{lemma9}, $D$ has a quasikernel
$Q_A\subseteq A$ with
\[
|Q|\le \frac{n+p_A}{3} \le \frac{n+\alpha n}{3}.
\]
The same argument applies if $p_B\le \alpha n$.

If $p_A> \alpha n$ and $p_B> \alpha n$,
applying Lemma \ref{claim3}, $D$ has a quasikernel $Q$ with
\[
|Q|\le\frac{7n}{5}-p_A-p_B\le \frac{7n}{5}-2\alpha n.
\]
Let $\alpha=16/35$.
Then in each case $D$ has a quasikernel of size at most $17n/35$.
\end{proof}
\section{Proof of Theorem \ref{th1}}

Let $D$ be a digraph. For a set $S\subseteq V(D)$, we write $\chi(S):=\{\chi(x):x\in S\}$ for the set
of colors appearing on $S$.

\begin{defi}
Let $r,L$ be positive integer with $L\ge r$.
An coloring $\chi: V(D) \to [r]$ of a digraph $D$ is called
\emph{$(r,L)$-locally colorful} on $D$ if for every $v\in V(D)$ and every
$0\le a\le r-1$,
\[
|\chi(B_D^-(v,L-a))|\ge r-a,
\]
where $|\chi(S)|$ denotes the number of distinct colors appearing in $S$.
A digraph $D$ is said to \emph{admit} an $(r,L)$-locally colorful coloring
if such a coloring exists.
\end{defi}

The following fact is an immediate consequence of the definition.

\begin{fact}\label{fact1}
Let $H$ be a spanning subgraph of a digraph $D$.
If $H$ admits an $(r,L)$-locally colorful coloring. Then the same vertex coloring $\chi$ is also $(r,L)$-locally colorful on $D$.
\end{fact}

The following fact is a basic result in digraph theory and can be found in \cite{BJG}.

\begin{fact}\label{fact2}
For every digraph $D$, the condensation digraph $\operatorname{con}(D)$ satisfies the following:
\begin{enumerate}
\item[(i)] $\operatorname{con}(D)$ is acyclic;
\item[(ii)] $\operatorname{con}(D)$ has at least one source vertex;
\item[(iii)] $\operatorname{con}(D)$ has at least one sink vertex;
\item[(iv)] $\operatorname{con}(D)$ admits a topological ordering of its vertices.
\end{enumerate}
\end{fact}

We need the following lemmas.

\begin{lem}\label{lemma1}
A digraph $D$ has no $(r-1)$-source set if and only if every source
strongly connected component of $D$ has at least $r$ vertices.
\end{lem}

\begin{proof}
 Suppose there is a source strongly connected component $C$ with fewer than $r$ vertices. By definition, we have $N^-(C)=\emptyset$, which implies that $C$ is an $(r-1)$-source set.

Conversely, assume that $S$ is a $(r-1)$-source set, which leads to $|S|\le r-1$. We claim that $S$ is a union of strongly connected components. In fact, if a strongly connected component  meets both $S$ and $V(D)\setminus S$, then $N_D^-(S)\ne\emptyset$, a contradiction.
Since $\operatorname{con}(D[S])$ has a source vertex and $S$ is a source set,
$S$ must contain a strongly connected component $C$ that has in-degree $0$ in $\operatorname{con}(D)$ with $|C|\le |S|\le r-1$,
which implies that $D$ contains a source strongly connected component with less than $r$ vertices.
\end{proof}

We now recall the definition of the ear decomposition; see \cite{R,W}.

\begin{defi}\label{def6}
Let $H$ be a subgraph of a digraph $D$. An \emph{ear} of $H$ in $D$ is a directed path
\[
P: x_0 \to x_1 \to \cdots \to x_\ell
\]
such that
 \begin{itemize}
 \item[(i)] $x_0, x_\ell \in V(H)$;
  \item[(ii)] $x_1,x_2,\dots, x_{\ell-1} \in V(D) \setminus V(H)$;
  \item[(iii)] $\ell \ge 1$.
  \end{itemize}
The vertices $x_0$ and $x_\ell$ are called the \emph{end vertices} of the ear and the vertices $x_1,x_2,\dots, x_{\ell-1}$ are called the \emph{internal vertices} of the ear.
Note that the end vertices are allowed to coincide.
\end{defi}
The following lemma is straightforward.
\begin{lem}\label{lemma2}
Let $D$ be a strongly connected digraph and let $H$ be a non-empty strongly connected subgraph of $D$.
If $H\neq D$, then $D$ contains an ear $P: x\to z_1\to\cdots\to z_k\to y$ of $H$ with $x,y\in V(H)$  the end vertices (not necessarily distinct) and
$z_1,\dots,z_k\in V(D)\setminus V(H)$  the internal vertices.
Consequently, the digraph obtained from $H$ by adding the ear $P$ remains strongly connected.
\end{lem}

\begin{lem}\label{lemma3}
Let $H$ be a subgraph of $D$ that admits an $(r,L)$-locally colorful coloring $\chi$.
Suppose
\[
x\to z_1\to\cdots\to z_k
\]
is a directed path such that $x\in V(H)$ and $z_1,\dots,z_k\in V(D)\setminus V(H)$.
Let $H'$ be the digraph obtained by adding this path to $H$.
Then $\chi$ can be extended to $z_1,\dots,z_k$ so that it is still an $(r,L)$-locally colorful coloring of $H'$.
\end{lem}

\begin{proof}
For each $0\le s\le r-1$, let $M_s=[r]\setminus \chi(B_H^-(x,L-s))$.
Thus $M_s$ is the set of colors not visible in the closed backward $(L-s)$-ball of $x$.
Since $\chi$ is $(r,L)$-locally colorful, we have $|M_s|\le s$ for all $0\le s\le r-1$ and $M_0\subseteq M_1\subseteq\cdots\subseteq M_{r-1}$.
We may relabel the colors, so that $M_s\subseteq\{1,2,\dots,s\}$. We now extend the coloring $\chi$ by defining
\[
\chi(z_j)=1+\big((j-1)\bmod r\big).
\]

We claim that $\chi$ is an $(r,L)$-locally colorful coloring of $H'$.
Fix $0\le a\le r-1$. For every $u\in V(H)$, it is clear that $B_H^-(u,L-a)=B_{H'}^-(u,L-a)$, which implies $$|\chi(B_{H'}^-(u,L-a))|=|\chi(B_H^-(u,L-a))|\ge r-a.$$
Now we consider $|\chi(B_{H'}^-(z_j,L-a))|$. If $j\ge r-a$, then $$\{z_{j-(r-a)+1},z_{j-(r-a)+2},\dots,z_j\}\subseteq B_{H'}^-(z_j,L-a),$$
which implies $|\chi(B_{H'}^-(z_j,L-a))|\ge r-a$. If $j<r-a$, then $$B_{H'}^-(z_j,L-a)=B_H^-(x,L-a-j)\cup \{z_1,z_2,\dots,z_j\}.$$
By $M_{a+j}\subseteq [a+j]$ and $\chi(B_H^-(x,L-a-j))=[r]\setminus M_{a+j}$,
we have $$([r]\setminus [a+j])\cup [j]\subseteq \chi(B_H^-(x,L-a-j))\cup \chi(\{z_1,z_2,\dots,z_j\})=\chi(B_{H'}^-(z_j,L-a)),$$ which leads to $|\chi(B_{H'}^-(z_j,L-a))|\ge r-a$.
Therefore, $\chi$ is still $(r,L)$-locally colorful in $H'$.
\end{proof}

By Lemma \ref{lemma3} and Fact \ref{fact1}, the same conclusion holds when the path is closed into an ear by adding a final edge to a vertex of $H$. We record this as the following corollary.

\begin{corollary}\label{cor1}
Let $H$ be a subgraph of $D$ that admits an $(r,L)$-locally colorful coloring $\chi$. If an ear
\[
x\to z_1\to\cdots\to z_k\to y
\]
with $x,y\in V(H)$ and $z_1,\dots,z_k\in V(D)\setminus V(H)$ is added to $H$, then $\chi$ can be extended to $z_1,\dots,z_k$ so that the resulting digraph still admits an $(r,L)$-locally colorful coloring.
\end{corollary}

\begin{lem}\label{lemma4}
If $D$ be a strongly connected digraph with $|V(D)|\ge r\ge 2$.  Then $D$
has a spanning subgraph admitting an $(r,3r-4)$-locally colorful coloring.
\end{lem}

\begin{proof}
We first construct a subgraph $H_0\subseteq D$ which admits an $(r,3r-4)$-locally colorful coloring.
If $D$ contains a directed cycle of length at least $r$, say $v_1\to v_2\to\cdots\to v_m\to v_1$, then we color the cycle by
$\chi(v_i)=1+((i-1)\bmod r)$. Let $H_0$ be the cycle. Then $\chi$ is an $(r,3r-4)$-locally colorful coloring on $H_0$.

 Now suppose every directed cycle in $D$ has length less than $r$. Since $D$ is strongly connected, it contains a directed cycle $C$. Starting from $C$ and repeatedly applying Lemma \ref{lemma2}, we can add ears successively to obtain a larger strongly connected subgraph. The process stops once the number of vertices increases to at least $r$. At each step, the current subgraph is strongly connected, so each added ear has at most $r-2$ internal vertices; otherwise, the resulting subgraph would contain a directed cycle of length at least $r$, contradicting the assumption. Consequently, we obtain a strongly connected subgraph $H_0$ satisfying $$r\le |V(H_0)|\le 2r-3.$$
Color $H_0$ with all $r$ colors appearing at least once.  Since $H_0$ is strongly connected, we have $B_{H_0}^-(x,2r-4)=H_0$ and $|\chi(B_{H_0}^-(x,2r-4))|=|\chi(H_0)|=r$.
Clearly, for any $0\le a\le r-1$ we have
\[
|\chi(B_{H_0}^-(x,3r-4-a))|=|\chi(H_0)|=r\ge r-a.
\]

In either case, we obtain an $(r, 3r-4)$-locally colorful colored strongly connected subgraph  $H_0$, with whose coloring denoted by $\chi$.
Thus, by Lemma \ref{lemma2}, starting from the strongly connected single-vertex subgraph
$H_0$, we may add ears successively to obtain a strongly connected
spanning subgraph of $D$. By Corollary \ref{cor1}, the coloring $\chi$ can be extended
to all vertices of $D$ while preserving the $(r,3r-4)$-locally colorful property. It follows from Fact 2.2 that $\chi$ is also an $(r,3r-4)$-locally colorful coloring on $D$, since $H_1 \subseteq D$.
\end{proof}

\begin{lem}\label{lemma5}
If $D$ has no $(r-1)$-source set, then $D$ has a spanning subgraph
admitting an $(r,3r-4)$-locally colorful coloring.
\end{lem}

\begin{proof}
Consider the condensation digraph $\operatorname{con}(D)$ of $D$. By Fact \ref{fact2}, the vertices of $\operatorname{con}(D)$ admit a topological ordering. We color the strongly connected components of $D$ according to this topological order.

Since $D$ has no $(r-1)$-source set, by Lemma \ref{lemma1}, every source strongly connected component contains at least $r$ vertices. Combining this with Lemma \ref{lemma4}, every source strongly connected component admits an $(r,3r-4)$-locally colorful coloring. We apply such a coloring to each source strongly connected component, and denote the resulting coloring by $\chi$.

We now color the non-source strongly connected components following the same topological ordering. Let $C$ be such a component at the time it is processed. By the definition of the topological ordering, there exists an arc $xy \in A(D)$ such that $x$ has already been colored and $y \in V(C)$. By Lemma \ref{lemma3}, the coloring $\chi$ can be extended to $y$ so that it remains $(r,3r-4)$-locally colorful. Applying Lemma \ref{lemma2} and Corollary \ref{cor1}, there is a spanning graph of $D[C]$ obtained from a strongly connected single vertex graph $\{y\}$ by adding ears. So the coloring $\chi$ can be extended to all vertices of $C$ while preserving the $(r,3r-4)$-locally colorful property.

After all components have been processed, we obtain a spanning subgraph $H$ of $D$ such that $\chi$ is an $(r,3r-4)$-locally colorful coloring of $H$. By Fact \ref{fact1}, $\chi$ is also an $(r,3r-4)$-locally colorful coloring of $D$.
\end{proof}

Now we are ready to present the proof of Theorem \ref{th1}.

\begin{proof}[Proof of Theorem \ref{th1}]
By Lemma \ref{lemma5}, $D$ admits an $(r,3r-4)$-locally colorful coloring  $\chi$. Let
\[
\chi: V(D) \to [r] \quad\text{ and } \quad A_i = \chi^{-1}(i) \text{~~for~~}1 \le i \le r.
\]
Taking $a=0$ in the definition of locally colorful, every vertex of $D$ is reachable within distance $3r-4$ from some vertex of each color class $A_i$.

By Theorem \ref{qk}, each digraph $D[A_i]$ contains a quasikernel, say $Q_i$. Then, for any $v \in V(D)$, there exist $a_i \in A_i$ and $q_i \in Q_i$ such that $\operatorname{dist}_D(a_i, v) \le 3r-4$ and $\operatorname{dist}_{D[A_i]}(q_i, a_i) \le 2$. Therefore,
\[
\operatorname{dist}_D(q_i, v) \le 3r-2.
\]
Thus each $Q_i$ is a $(3r-2)$-kernel of $D$, and the $Q_i$ are pairwise disjoint.
\end{proof}

\section{Proof of Theorem \ref{th2}}

We first recall the following structural lemma of Spiro \cite{S}, which describes the elementary structure of source-free unicyclic digraphs.

\begin{lem}\label{lemma6}\cite{S}
If $D$ is unicyclic, then
\begin{itemize}
\item[(i)] $D$ contains exactly one directed cycle $C$;
\item[(ii)] for every vertex $w \in V(D)$, there exists a unique directed path $P_w$ whose first vertex is in $C$, last vertex is $w$, and which contains no other vertices of $C$.
\end{itemize}
\end{lem}
Now we present the proof of Theorem \ref{th2}.

\begin{proof}[ Proof of Theorem \ref{th2}.]
Applying the structural description in Lemma \ref{lemma6}  to the digraph  $D$ in  Theorem \ref{th2}, the unique directed cycle of $D$ can be written as
\[
u_0 \to v_0 \to u_1 \to v_1 \to \cdots \to u_{\ell-1} \to v_{\ell-1} \to u_0,
\]
with \(u_i \in U\), \(v_i \in V\) for $i=0,1,\ldots,\ell-1$.

We now define two coordinates $\rho_U, \rho_V: V(D) \to \mathbb{Z}_\ell$. On the directed cycle, we set
\[
\left\{
\begin{aligned}
\rho_U(u_i) &= \rho_U(v_i) = i, \\
\rho_V(v_i) &= i,\quad \rho_V(u_i) = i-1
\end{aligned}
\right.
.
\]
for each $ i \in \mathbb{Z}_\ell$.
Off the cycle, we extend $\rho_U$ and $\rho_V$ along the unique directed paths from the cycle by
\[
\left\{
\begin{aligned}
\rho_U(y) &= \rho_U(x) \ \text{if } x \to y,\ x \in U,\ y \in V, \quad
\rho_U(y) = \rho_U(x) + 1 \ \text{if } x \to y,\ x \in V,\ y \in U; \\[0.5em]
\rho_V(y) &= \rho_V(x) \ \text{if } x \to y,\ x \in V,\ y \in U, \quad
\rho_V(y) = \rho_V(x) + 1 \ \text{if } x \to y,\ x \in U,\ y \in V.
\end{aligned}
\right.
\]
These extensions are well-defined by Lemma \ref{lemma6}, which guarantees that every vertex off the cycle lies on a unique directed path from the cycle.
Let $h=(q+1)/2$, $k=\left\lceil\ell/h\right\rceil$, and $$S(a)=\{a,a+h,a+2h,\dots,a+(k-1)h\}\subseteq \mathbb{Z}_\ell.$$ Define
\[
Q_U(a)=\{x\in U:\rho_U(x)\in S(a)\},\quad Q_V(a)=\{x\in V:\rho_V(x)\in S(a)\}.
\]

Now we prove that $Q_U(a)$ and $Q_V(a)$ are $q$-kernels of $D$.
By symmetry, it suffices to prove that $Q_U(a)$ is a $q$-kernel of $D$.
Since $Q_U(a) \subseteq U$ and $D$ has bipartition $U \cup V$, the set $Q_U(a)$ is independent.

Every vertex of $Q_U(a)$ is reached from $Q_U(a)$ in distance $0$. Let $x\in V(D)\setminus Q_U(a)$. Since $(k-1)h<\ell\le kh$, the set $S(a)$ has size $k$, and its gaps around the cyclic order of $\mathbb Z_\ell$ are all at most $h$. Hence one of the residues $\rho_U(x),\rho_U(x)-1,\ldots,\rho_U(x)-(h-1)$ belongs to $S_a$, say $\rho_U(x)-d$, where $d\in \{0,1,2,\dots,h-1\}$.

By the definition, $\rho_U(x),\rho_U(x)-1,\ldots,\rho_U(x)-(h-1)$ are exactly the $\rho_U$-coordinates of the first $h$ vertices of $U$ on the unique backward in-neighbor chain from $x$  counted along
the chain (modulo $\ell$). So we can select a vertex $y$ in this backward in-neighbor chain with $y\in U$ and $\rho_U(y)=\rho_U(x)-d$, which leads to $y\in Q_U(a)$.

The backward in-neighbor chain from $x$ to $y$ is a directed walk from $y$ to $x$. Its length is at most $2d$ if $x\in U$, and at most $2d+1$ if $x\in V$. Since $q=2h-1$, we have $\operatorname{dist}_D(Q_U(a),x)\le q$. Therefore $Q_U(a)$ is a $q$-kernel of $D$.

As $a$ ranges over $\mathbb{Z}_\ell$, each residue of $\mathbb{Z}_\ell$ belongs to exactly $k$ of the sets $S(a)$.
\[
\sum_{a\in \mathbb{Z}_\ell}|Q_U(a)|=k|U|,\qquad \sum_{a\in \mathbb{Z}_\ell}|Q_V(a)|=k|V|.
\]
Therefore there exist \(a,b\in \mathbb{Z}_\ell\) such that
\[
|Q_U(a)|\le \frac{k}{\ell}|U|,\qquad |Q_V(b)|\le \frac{k}{\ell}|V|,
\]
which leads to
\[|Q_U(a)|+|Q_V(b)|\le\frac{k}{\ell}(|U|+|V|)=\frac{\lceil \ell/h\rceil}{\ell}|V(D)|\le  2\cdot\frac{\lceil \ell/(q+1)\rceil}{\ell}|V(D)|.
\]
\end{proof}
\section{Declaration on the Use of AI}
At an early stage of this work, AI-assisted tools were used as a brainstorming aid.   All results and proofs were written by the authors, who take full responsibility for the content of this work.

\end{document}